\documentclass[12pt]{article}
\usepackage{amssymb,latexsym}
\usepackage{amsmath}
\usepackage{rawfonts}
\usepackage{amsmath, psfrag}
\usepackage{amsfonts}
\usepackage{amsthm}

\newtheorem{theorem}{Theorem}

\def\qed{\ifhmode\unskip\nobreak\fi\quad\ifmmode\Box\else$\Box$\fi}

\def\qed{\hfill $\Box$}

\begin{document}
\title{Order plus size of $\tau$-critical graphs}
\author{
 {\sl Andr\'as Gy\'arf\'as}\\
 Alfr\'ed R\'enyi Mathematical\\ Institute,
Budapest, Hungary
 \and
{\sl Jen\H{o} Lehel} \\ University of Louisville\\ Louisville, KY 40292}
\maketitle

\begin{abstract} Let $G=(V,E)$ be a $\tau$-critical graph with $\tau(G)=t$. Erd\H{o}s and Gallai 
 proved that $|V|\leq 2t$ and the bound $|E|\leq {t+1\choose 2}$ was obtained by
Erd\H{o}s, Hajnal and Moon.
We give here the sharp combined bound $|E|+|V|\leq {t+2\choose 2}$ and find all graphs with equality.
\end{abstract}

A set of vertices meeting every edge of a graph $G$ is called a {\it transversal set} of $G$. The {\it transversal number} of $G$, $\tau(G)$,  is defined to be the
the minimum cardinality of a transversal set of $G$.
A simple graph $G=(V,E)$ with no isolated vertex is called {\it $\tau$-critical} if $\tau(G-e)=\tau(G)-1$, for every $e\in E$ (where $G-e=(V, E\setminus\{e\})$). The primary sources for the properties of $\tau$-critical graphs are
Lov\'asz and Plummer \cite[Chapter 12.1]{LP}, and
Lov\'asz \cite[Chapter 8, Exercises 10--25]{LO}.

The tight bounds for the number of edges and the number of vertices in a $\tau$-critical graph $G=(V,E)$ with $\tau(G)=t$
are:
\begin{equation}
\label{bounds} |V|\leq 2t
\quad \hbox{ and } \quad |E|\leq {t+1\choose 2}.
\end{equation}

The vertex bound is due to Erd\H{o}s and Gallai \cite{EG}, and the edge bound was obtained by Erd\H{o}s, Hajnal, and Moon \cite{EHM}. Here we derive the combined bound
$|E|+|V|\leq{t+2\choose 2}$ and determine all extremal graphs (Theorem \ref{main}). Note that the combined bound immediately gives the edge bound in  (\ref{bounds}) since $|V|\ge t+1$. The proof of the combined bound comes easily from the next degree bound.

\medskip

\noindent{\bf Theorem A. [Hajnal \cite{H}]}
\label{maxdegree}
{\it Let $G=(V,E)$ be a $\tau$-critical graph
of order $n$
with $\tau(G)=t$. Then $deg_G(x)\leq 2t-n+1$ for every $x\in V$. }\qed \\

\begin{theorem}
\label{main}
 If
 $G=(V,E)$ is a $\tau$-critical graph of order $n$ with $\tau(G)=t$, then
 \begin{equation}
 \label{vemax}
 |V|+|E|\leq {t+2\choose 2}.
  \end{equation}
 Furthermore, the bound is tight if and only if $G\cong K_{n}$, $n\geq 2$, or $G\cong 2K_2$ or $G\cong C_5$.
\end{theorem}

\begin{proof}
By Theorem A, we have
  \begin{equation}
\label{compute}
|V|+|E|\le n+ {n(2t-n+1)\over 2},
\end{equation}
with equality if $G$ is a $(2t-n+1)$-regular graph. To prove (\ref{vemax}), we show that the right hand side of (\ref{compute}) is at most ${t+2\choose 2}$. This 
is equivalent to
$0\le (n-t-1)(n-t-2)$ which is clearly true, since $n\ge t+1$, with equality 
only for $n=t+1$ or $n=t+2$. 
Thus equality in (\ref{vemax}) is possible only for $(n-1)$-regular and for $(n-3)$-regular graphs.

In the first case $G=K_{n}$.
In the second case the candidates are the graphs whose complements are $2$-regular (and have at least four vertices).
Since these graphs are $\tau$-critical with $t=n-2$, the deletion of any edge creates a set of three vertices inducing no edges; equivalently,  including an edge in their complements produces a triangle. This implies that the complement of such a graph $G$ must be a single cycle $C$, since otherwise, deletion of an edge between two cycles creates no triangle. In addition, $C$ has at most five vertices because deletion of a long diagonal would not create a triangle. Thus $C$ is a four-cycle (and then $G=2K_2$), or $C$ (and its complement $G$) is a five cycle. 
\end{proof}


\begin{thebibliography}{99}

\bibitem{EG} P. Erd\H{o}s and T. Gallai, On the maximal number of vertices representing the edges of a graph, K\"ozl. MTA Mat. Kutat\'o Int. Budapest 6 (1961) 181--203.

\bibitem{EHM} P. Erd\H{o}s, A. Hajnal, and J. W. Moon: A problem in graph theory, Amer. Math. Monthly 71 (1964) 1107--1110.

\bibitem{H} A. Hajnal, A theorem on $k$-saturated graphs,
 Canadian Journal of Math. 17 (1965) 720--724.

\bibitem{LO} L. Lov\'asz,
{\bf Combinatorial problems and exercises}, Second Edition,
  {AMS Chelsea Publishing, Providence, RI},   {2007}.
		
\bibitem{LP} L. Lov\'asz,  and M.D. Plummer, {\bf Matching Theory}. Akad\'emiai Kiad\'o, North Holland 1986.


\end{thebibliography}
\end{document}